\RequirePackage{tikz}
\documentclass[sn-mathphys-num]{sn-jnl}

\usepackage{graphicx}%
\usepackage{multirow}%
\usepackage{amsmath,amssymb,amsfonts}%
\usepackage{amsthm}%
\usepackage{mathrsfs}%
\usepackage[title]{appendix}%
\usepackage{xcolor}%
\usepackage{textcomp}%
\usepackage{manyfoot}%
\usepackage{booktabs}%
\usepackage{algorithm}%
\usepackage{algorithmicx}%
\usepackage{algpseudocode}%
\usepackage{listings}%

\usepackage[utf8]{inputenc}
\usepackage[T1]{fontenc}
\usepackage{geometry}
\usepackage{microtype}
\usepackage{subcaption}
\usepackage{bbm}
\usepackage{bm}
\usepackage{url}
\usepackage{multirow}
\usepackage{setspace}
\usepackage{mathtools} 
\usepackage{xcolor}
\usepackage{xspace}

\usepackage{booktabs}
\usepackage{csquotes}
\usepackage{soul}
\usepackage{environ} 
\usepackage{trimclip}
\renewcommand{\b}[1]{\mathbf #1}

\newcommand{\ie}{i.\,e.\ }
\newcommand{\eg}{e.\,g.\ }
\newcommand{\dt}{\Delta t}

\newcommand{\bsigma}{\boldsymbol{\sigma}}

\renewcommand{\O}{\mathcal O}

\newcommand{\R}{\mathbb R}

\newcommand{\qta}{\quad\text{ and }\quad}

\DeclareMathOperator{\NB}{NB}

\usepackage{forest}
\forestset{
	declare toks={elo}{}, 
	anchors/.style={grow'=90, anchor=#1,child anchor=#1,parent anchor=#1},
	dot/.style={tikz+={\fill (.child anchor) circle[radius=#1];}},
	dot/.default=2pt,
	decision edge label/.style n args=3{
		edge label/.expanded={node[midway,auto=#1,anchor=#2,\forestoption{elo}]{\strut$\unexpanded{#3}$}}
	},
	decision/.style={if n=1
		{decision edge label={left}{east}{#1}}
		{decision edge label={right}\label{key}{west}{#1}}
	},
	decision tree/.style={
		for tree={grow'=90,
			s sep=2.5pt, l=0pt, l sep =0.5pt, outer sep =-1.5pt,
			if n children=0{anchors=west}{
				if n=1{anchors=west}{anchors=west}},
			math content,
		},
		anchors=west, outer sep=-1.5pt,
		dot=2pt, for descendants=dot,
		delay={for descendants={split option={content}{;}{content,decision}}},
	},
	rooted tree/.style={
		for tree={
			grow'=90,
			parent anchor=center,
			child anchor=center,
			s sep=2.5pt,
			l sep =1pt,
			if level=0{
				baseline
			}{},
			delay={
				if content={*}{
					content=,
					append={[]}
				}{}
			}
		},
		before typesetting nodes={
			for tree={
				circle,
				fill,
				minimum width=3pt,
				inner sep=0pt,
				child anchor=center,
			},
		},
		before computing xy={
			for tree={
				l=5pt,
			}
		}
	}
}

\DeclareDocumentCommand\ct{o}{\Forest{decision tree [#1]}}
\DeclareDocumentCommand\rt{o}{\Forest{rooted tree [#1]}}

\renewcommand{\tt}[2]{\resizebox{#1 cm}{!}{\marginbox*{0pt -0.25\height pt 0pt 0pt}{#2}}}

\newcommand{\bssigma}{\boldsymbol{\sigma}}

\newcommand{\byi}{\b y^{(i)}}
\newcommand{\byj}{\b y^{(j)}}

\newcommand{\yi}{y^{(i)}}

\newcommand{\fnu}{\mathbf{f}^{[\nu]}}

\newcommand{\dF}{\mathcal{F}}

\usepackage{tikz}
\usepackage{tkz-euclide} 
\usetikzlibrary{patterns}

\definecolor{colorA}{rgb}{0,0.447,0.741}
\definecolor{colorB}{rgb}{0.85,0.325,0.098}
\definecolor{colorE}{rgb}{0.929,0.694,0.125}
\definecolor{colorF}{rgb}{0.494,0.184,0.556}
\definecolor{colorD}{rgb}{0.466,0.674,0.188}
\definecolor{colorC}{rgb}{0.301,0.745,0.933}
\definecolor{colorG}{rgb}{0.635,0.078,0.184}

\newlength{\tickl}    
\tikzset{axes/.style={thick,-latex}}
\tikzset{lineplot/.style={thick}}
\tikzset{arrow/.style={thick,-latex}} 
\tikzset{thick arrow/.style={ultra thick,-latex}} 
\tikzset{grid lines/.style={very thin,gray!30}}	
\tikzset{point/.style={radius=2pt}}
\tikzset{help line/.style={black,thin,dashed}} 

\makeatletter
\newsavebox{\measure@tikzpicture}
\NewEnviron{scaletikzpicturetowidth}[1]{%
	\def\tikz@width{#1}%
	\begin{lrbox}{\measure@tikzpicture}%
		\BODY
	\end{lrbox}%
	\pgfmathparse{#1/\wd\measure@tikzpicture}%
	\BODY
}
\makeatother
\usetikzlibrary{decorations.markings}
\usepackage{pgfplots}

\tikzset{mylabel/.style  args={at #1 #2  with #3}{
		postaction={decorate,
			decoration={
				markings,
				mark= at position #1
				with  \node [#2] {#3};
} } } }

\theoremstyle{thmstyleone}%
\newtheorem{theorem}{Theorem}
\newtheorem{lemma}[theorem]{Lemma}
\theoremstyle{thmstyletwo}%

\theoremstyle{thmstylethree}%
\newtheorem{definition}{Definition}%

\begin{document}

\title[Dense Output for MPRK Schemes]{A Boot-Strapping Technique to Design Dense Output Formulae for Modified Patankar--Runge--Kutta Methods}


\author*[1]{\fnm{Thomas} \sur{Izgin}}\email{izgin@mathematik.uni-kassel.de}

%

\affil*[1]{\orgdiv{Department of Mathematics}, \orgname{University of Kassel}, \orgaddress{\street{Heinrich-Plett-Str. 40}, \city{Kassel}, \postcode{34132}, \state{Hessen}, \country{Germany}}, ORCID:~0000-0003-3235-210X}




\abstract{In this work modified Patankar--Runge--Kutta (MPRK) schemes up to order four are considered and equipped with a dense output formula of appropriate accuracy. Since these time integrators are conservative and positivity preserving for any time step size, we impose the same requirements on the corresponding dense output formula. In particular, we discover that there is an explicit first order formula. However, to develop a boot-strapping technique we propose to use implicit formulae which naturally fit into the framework of MPRK schemes. In particular, if lower order MPRK schemes are used to construct methods of higher order, the same can be done with the dense output formulae we propose in this work. We explicitly construct formulae up to order three and demonstrate how to generalize this approach as long as the underlying Runge--Kutta method possesses a dense output formulae of appropriate accuracy. 
	
We also note that even though linear systems have to be solved to compute an approximation for intermediate points in time using these higher order dense output formulae, the overall computational effort is reduced compared to using the scheme with a smaller step size.   }




\maketitle

\section{Introduction}
The first \emph{modified Patankar--Runge--Kutta} (MPRK) method was introduced in 2003 based on the explicit Euler method \cite{BDM03}. The resulting \emph{modified Patankar--Euler} (MPE) method is proven to be first order accurate, \emph{unconditionally positive} and \emph{conservative}. Unconditional positivity means that the method produces positive approximations for all time step sizes $\dt>0$ whenever the initial data is positive. Additionally, conservativity means that the sum of all constituents of the numerical approximation in any time step and for any $\dt>0$ equals the sum of the constituents of the initial data. While these two properties are also guaranteed by the implicit Euler method, the advantage of the MPE scheme is that it only requires the solution of a linear system of equations at each time step, even for nonlinear differential equations. Furthermore, unconditional positivity for linear methods such as Runge--Kutta (RK) schemes can only be guaranteed by a first order scheme \cite{Sandu02,BC78}. However, MPRK methods do not fall into this class of methods as they are nonlinear even for linear problems, see for example \cite{IKM2122}. Indeed, besides second and third order MPRK schemes \cite{KM18, KM18Order3}, there are even arbitrary high order modified Patankar-type (MP) schemes based on Deferred Correction (MPDeC) methods \cite{MPDeC}, all of which are unconditionally positive. Because of the nonlinearity of these methods a stability analysis and a comprehensive framework for deriving order conditions was only developed recently \cite{izgin2022stability,NSARK}, see also \cite{IzginThesis} for an overview on Patankar-type schemes and their analysis. 

It is worth noting that MPRK schemes based on an $s$-stage RK method require the solution of at least $s$ linear systems in each time step. Hence, it is worth reducing the computational cost by, for instance, equipping the methods with a time step controller, which is done in \cite{IR20233}. Another way to reduce the overall computational cost is the design of a \emph{dense output formula}  \cite{EJNT1986}. The idea of such a formula, also known as \emph{contiunous extension} \cite{Zennaro1986} is to obtain an approximation of the same order of convergence at any given point in time from the numerical approximation at finite times and a comparably small additional computational cost. However, since the unique selling point of MPRK schemes is to be unconditionally positive and conservative, the same requirements should be applied to the dense output formula. 

It is also common to use lower order dense output formulae to construct one of higher order. The corresponding algorithm is called \emph{boot-strapping process} \cite{EJNT1986}. However, besides MPDeC there are currently only MPRK schemes up to order four known. Still, in view of this active research field a boot-strapping process for even higher order MPRK schemes is of interest. Altogether, designing the first dense output formulae for all MPRK schemes up to order four, and developing such a boot-strapping process for higher order MPRK methods is the purpose of the present work. 

In the upcoming section we first briefly introduce MPRK schemes and present preliminary results which are needed in this work. We then construct a first order dense output formula starting the boot-strapping process and elaborate several approaches for constructing higher order dense output formulae discussing their unconditional positivity and conservativity. Finally, we present the boot-strapping process together with formulae for schemes up to fourth order and conclude this work with a summary and an outlook. 

\section{Preliminaries}
Modified Patankar--Runge--Kutta (MPRK) schemes were originally introduced to approximate the solution of a so-called \emph{positive} and \emph{conservative} autonomous \emph{production-destruction system} (PDS)
\begin{equation} \label{eq:PDS}
	y_k'(t)= f_k(\b y(t))=\sum_{\nu=1}^N (p_{k\nu}(\b y(t))-d_{k\nu}(\b y(t))),\quad k=1,\dotsc,N,\quad \b y(0)=\b y^0>\b 0
\end{equation}
with $p_{k\nu}(\b y),d_{k\nu}(\b y)\geq 0$ for $\b y>\b 0$ (componentwise). Here, positivity means that $\b y(0)>\b 0$ implies $\b y(t)>\b 0$ for all $t>0$ and conservativity means $p_{k\nu}=d_{\nu k}$ for all $k,\nu=1,\dotsc,N$.

\begin{definition}\label{def:MPRKdefn}
	Given an explicit $s$-stage RK method described by a non-negative Butcher array, \ie $\b A,\b b,\b c \geq\b 0$ we define the corresponding MPRK schemes applied to \eqref{eq:PDS} by
	\begin{equation}\label{eq:MPRK}
	\begin{aligned} 
		\yi_k & = y^n_k + \dt \sum_{j=1}^{i-1} a_{ij} \sum_{\nu=1}^N \left( p_{k\nu}(\byj) \frac{\yi_\nu}{\pi_\nu^{(i)}} - d_{k\nu}(\byj)\frac{\yi_k}{\pi_k^{(i)}} \right),  \quad i=1,\dotsc,s, \\
		y^{n+1}_k & = y^n_k + \dt \sum_{j=1}^s b_j \sum_{\nu=1}^N \left( p_{k\nu}(\byj)\frac{y^{n+1}_\nu}{\sigma_\nu} - d_{k\nu}(\byj) \frac{y^{n+1}_k}{\sigma_k} \right),\quad  k=1,\dotsc,N.
	\end{aligned}
\end{equation}
	where $\pi_\nu^{(i)},\sigma_\nu$ are the so-called \emph{Patankar-weight denominators} (PWDs) and positive for any $\dt\geq 0$ as well as independent of the corresponding numerators $y_k^{(i)}$ and $y_k^{n+1}$, respectively. 
\end{definition}
The unconditional positivity of these schemes is then proved by showing that the mass matrices defining the linear systems required to compute the stages and the update $\b y^{n+1}$ are $M$-matrices, \ie have non-negative inverses \cite{KM18}. Since Definition~\ref{def:MPRKdefn} does not specify the PWDs, one may ask what conditions need to be satisfied by them to ensure that the method is of a certain order. In what follows we briefly summarize the corresponding results of interest from \cite{NSARK}. 
\subsection{Order Conditions}
We want to emphasize here, that this section does not reflect all technical details discussed in \cite{NSARK} but rather gives the main ideas and results important for the present work. The main idea is to interpret MPRK schemes as additive Runge--Kutta (ARK) methods with a solution dependent Butcher tableau. This interpretation is valid, since a PDS \eqref{eq:PDS} represents a special additive splitting
\[\b f(\b y(t))=\sum_{\substack{\nu=1}}^N  \fnu(\b y(t))\]
of the right-hand side using
\[f^{[\nu]}_k(\b y(t))=\begin{cases}p_{k\nu}(\b y(t)), &k\neq \nu, \\- \sum_{\mu=1}^Nd_{k\mu}(\b y(t)), &k=\nu,\end{cases}\]
see \cite[Remark 2.25]{IzginThesis}.
For ARK methods, the framework for deriving order conditions is based on truncated NB-series \[\NB_p(u,\b y)=\b y+ \sum_{\tau\in NT_p}\frac{\dt^{\lvert \tau\rvert}}{\sigma(\tau)}u(\tau)\dF(\tau)(\b y).\]
Here, $\tau$ is a colored rooted tree \cite{Ntrees}, in which each node possesses one of $N$ possible colors from the set $\{1,\dotsc, N\}$. The set of all such \emph{$N$-trees} is denoted by $NT$, and the \emph{order} $\lvert\tau\rvert$  equals the number of its nodes. With that $NT_p$ is the set of all $N$-trees up to order $p$, where we set $NT_0\coloneqq \emptyset$ resulting in $\NB_0(u,\b y)=\b y.$ Furthermore, $\sigma$ is the \emph{symmetry} and $\dF$ represents an \emph{elementary differential}, see \cite{Ntrees} for the details. 

In general, a colored rooted tree $\tau$ with a root color $\nu$ can be written in terms of its colored \emph{children} $\tau_1,\dotsc,\tau_l$ by writing 
\begin{equation}\label{eq:tau}
	\tau=[\tau_1,\dotsc,\tau_l]^{[\nu]},
\end{equation}
where the children $\tau_1,\dotsc,\tau_l$ are the connected components of $\tau$ when the root together with its edges are removed. Moreover, the neighbors of the root of $\tau$ are the roots of the corresponding children.
Also, a tree with a single node and color $\nu$ is represented as $\rt[]^{[\nu]}$. 
%
%
The main idea is now to write both the analytic and numerical solution in terms of a truncated NB-series. Indeed, introducing the \emph{density} $\gamma$ for $\tau$ from \eqref{eq:tau} as
\[\gamma(\tau)=\lvert \tau\rvert \prod_{i=1}^l\gamma(\tau_i),\quad\gamma(\rt[]^{[\nu]})=1,\quad \nu=1,\dotsc,N,\] 
we have the following result.
\begin{theorem}[{\cite[Theorem 1]{Ntrees}}]\label{thm:anasolNB}
	Let $\fnu\in \mathcal C^{p+1}$ for $\nu=1,\dotsc,N$. Then the analytic solution $\b y$ can be written as
	\[\b y(t+\dt)=\NB_p(\tfrac{1}{\gamma},\b y(t)) +\O(\dt^{p+1}).\]
\end{theorem}
For the NB-series of the numerical solution, let us set 
\begin{equation}\label{eq:NSW}
	a_{ij}^{[\nu]}(\b y^n,\dt)=a_{ij}\frac{\yi_\nu}{\pi_{\nu}^{(i)}} \quad \text{ and } \quad b_j^{[\nu]}(\b y^n,\dt)=b_j\frac{y^{n+1}_\nu}{\sigma_\nu}
\end{equation}
for $i,j=1,\dotsc,s$ and $\nu=1,\dotsc,N$, where the dependence on $\b y^n$ and $\dt$ is given implicitly\footnote{We recall that while $\sigma$ is the symmetry, $\sigma_\nu$ with $\nu\in\{1,\dotsc, N\}$ or $\bsigma=(\sigma_1,\dotsc,\sigma_N)^T$ is a Patankar-weight denominator (vector).}. Next, following \cite{NSARK}, we introduce
\begin{equation}\label{eq:pertcond}
	\begin{aligned}
		u(\tau,\b y^n,\dt)&=\sum_{\nu=1}^N\sum_{i=1}^sb_i^{[\nu]}(\b y^n,\dt) g_i^{[\nu]}(\tau,\b y^n,\dt),\\ 
		g_i^{[\nu]}(\rt[]^{[\mu]},\b y^n,\dt)&=\delta_{\nu\mu},\quad \nu,\mu=1,\dotsc,N, \\
		g_i^{[\nu]}([\tau_1,\dotsc,\tau_l]^{[\mu]},\b y^n,\dt)&=\delta_{\nu\mu}\prod_{j=1}^ld_i(\tau_j,\b y^n,\dt),\quad \nu,\mu=1,\dotsc,N\text{ and } \\
		d_i(\tau,\b y^n,\dt)&=\sum_{\nu=1}^N\sum_{j=1}^sa_{ij}^{[\nu]}(\b y^n,\dt) g_j^{[\nu]}(\tau,\b y^n,\dt).
	\end{aligned}
\end{equation} 
The main result of \cite{NSARK} essentially states that $u$ from \eqref{eq:pertcond} is used for the NB-series of the numerical solution. Hence, comparing with Theorem~\ref{thm:anasolNB} it is shown that an MPRK scheme is of order $p$ if and only if
	\begin{equation}\label{eq:u=1:gamma}
	u(\tau, \b y^n,\dt)=\frac{1}{\gamma(\tau)}+\O(\dt^{p+1-\lvert \tau \rvert}), \quad \forall \tau\in NT_p.
\end{equation}
It is worth noting that these order conditions just equal the usual RK order conditions with two exceptions. First, the coefficients $a_{ij}$ and $b_j$ are replaced by the weighted ones from \eqref{eq:NSW}, and second, the order conditions tolerate a truncation error $\O(\dt^{p+1-\lvert \tau \rvert})$.

Another important observation from \cite{NSARK}, which will be used in this work is that if the MPRK scheme is of order $p$ then the PWD $\bssigma$ must be an $(p-1)$-th order approximation.
\begin{lemma}\label{lem:sigma}
	Let $\b A,\b b,\b c$ describe an explicit $s$-stage RK method of at least order $p$. Consider the corresponding MPRK scheme \eqref{eq:MPRK} and assume $p_{k\nu},d_{k\nu}\in \mathcal{C}^{p+1}$ for $k,\nu=1,\dotsc,N$. If the MPRK method is of order $p$, then
 	\[\bsigma=\NB_{p-1}(\tfrac1\gamma,\b y^n)+\O(\dt^{p}). \]
\end{lemma}
\section{Dense Output Formulae}
Since MPRK methods are based on explicit RK schemes, we first look at dense output formulae for these methods. An $s$-stage explicit RK method applied to \eqref{eq:PDS} reads
\begin{equation} \label{eq:RK}
	\begin{aligned}
	\yi_k & = y^n_k + \dt \sum_{j=1}^{i-1} a_{ij} \sum_{\nu=1}^N \left( p_{k\nu}(\byj) - d_{k\nu}(\byj) \right),  \quad i=1,\dotsc,s, \\
	y^{n+1}_k & = y^n_k + \dt \sum_{j=1}^s b_j \sum_{\nu=1}^N \left( p_{k\nu}(\byj) - d_{k\nu}(\byj)  \right),\quad  k=1,\dotsc,N.
\end{aligned}	
\end{equation}

A dense output formula now replaces $b_j\in \R$ by a function $\bar b_j\colon [0,1]\to \R$ such that
\begin{equation}\label{eq:RKdense}
	y^{n+\theta}_k = y^n_k + \dt \sum_{j=1}^s \bar b_j(\theta) \sum_{\nu=1}^N \left( p_{k\nu}(\byj) - d_{k\nu}(\byj)  \right)
\end{equation} 
is an approximation to $y_k(t^n+\theta\dt)$. With that in mind, it is natural to impose $\bar b_j(0)=0$ and $\bar b_j(1)=b_j$ to recover
\begin{equation}\label{eq:inner_consistency}
	\b y^{n+\theta}=\begin{cases}
		\b y^n,& \theta = 0,\\
		\b y^{n+1}, & \theta = 1.
	\end{cases}
\end{equation}

In the case of MPRK schemes, it thus seems to be natural to set 
\begin{equation}\label{eq:DO_ansatz1}
y^{n+\theta}_k = y^n_k + \dt \sum_{j=1}^s \bar b_j(\theta) \sum_{\nu=1}^N \left( p_{k\nu}(\byj)\frac{y^{n+1}_\nu}{\sigma_\nu} - d_{k\nu}(\byj) \frac{y^{n+1}_k}{\sigma_k} \right)
\end{equation}
to obtain an explicit dense output formula, which is applied after $\b y^{n+1}$ is computed. In this case, it is easily seen analogously to \eg \cite[Theorem 6.1]{HNW1993} that the analytic solution satisfies $\b y(t+\theta\dt)=\NB^\theta_p(\frac1\gamma,\b y)+\O(\dt^{p+1}),$
where \[\NB^\theta_p(u,\b y)\coloneqq\b y+ \sum_{\tau\in NT_p}\frac{(\theta\dt)^{\lvert \tau\rvert}}{\sigma(\tau)}u(\tau)\dF(\tau)(\b y).\] Hence, the dense output formula is of order $p^*$ if and only if 
	\begin{equation}\label{eq:u=theta:gamma}
		\begin{aligned}
	u(\tau, \b y^n,\dt,\theta)&=\frac{\theta^{\lvert \tau \rvert}}{\gamma(\tau)}+\O(\dt^{p^*+1-\lvert \tau \rvert}), \quad \forall \tau\in NT_{p^*}, \quad \text{where}\\
		u(\tau,\b y^n,\dt,\theta)&=\sum_{\nu=1}^N\sum_{i=1}^sb_i^{[\nu]}(\b y^n,\dt,\theta) g_i^{[\nu]}(\tau,\b y^n,\dt)
	\end{aligned}
\end{equation}
as in \eqref{eq:pertcond} with 
\begin{equation}\label{eq:bjtheta_exp}
	b_i^{[\nu]}(\b y^n,\dt,\theta)=\bar b_j(\theta)\frac{y^{n+1}_\nu}{\sigma_\nu}.
\end{equation}
Having stated the order conditions, it is worth noting that an MPRK method of order $p$ only needs to be equipped with a dense output formula of order $p^*=p-1$ to have an overall convergence rate of order $p$, see \cite[Section II.6]{HNW1993}. We also note that these conditions not guarantee positivity of \eqref{eq:DO_ansatz1}, while conservativity is naturally satisfied. Still, already in \cite{Zennaro1986,KLJK17} an unconditional positive and conservative dense output formula of order $p^*=1$ is constructed, which we present in the following section.

\subsection{First Order Dense Output}
For the construction of a first order formula, we directly use the condition \eqref{eq:u=theta:gamma} with the ansatz \eqref{eq:bjtheta_exp} for $p^*=1$, reading
\begin{equation}\label{eq:MPRKord1}
	u(\rt[]^{[\mu]},\b y^n,\dt,\theta)=\sum_{j=1}^s \bar b_j(\theta)\frac{y^{n+1}_\mu}{\sigma_\mu}=\theta +\O(\dt).
\end{equation}
Hence, using $\bar b_j(\theta)=\theta b_j$ representing a piecewise linear interpolant of the numerical data $\b y^n$, we see that $u(\rt[]^{[\mu]},\b y^n,\dt,\theta)=\theta u(\rt[]^{[\mu]},\b y^n,\dt)$. Hence, this choice of $\bar b_j$ yields a first order dense output formula according to \eqref{eq:u=1:gamma} for any MPRK method of order $p\geq 1$. 

To see positivity of the formula, it is beneficial to rewrite it as
	\begin{align}
		y^{n+\theta}_k &= y^n_k + \theta\dt \sum_{j=1}^s b_j \sum_{\nu=1}^N \left( p_{k\nu}(\byj)\frac{y^{n+1}_\nu}{\sigma_\nu} - d_{k\nu}(\byj) \frac{y^{n+1}_k}{\sigma_k} \right)\label{eq:first_order}\\
		&=(1-\theta)y^n_k+ \theta \left(y^n_k+\dt \sum_{j=1}^s b_j \sum_{\nu=1}^N \left( p_{k\nu}(\byj)\frac{y^{n+1}_\nu}{\sigma_\nu} - d_{k\nu}(\byj) \frac{y^{n+1}_k}{\sigma_k} \right)\right)\nonumber\\
		&=(1-\theta)y^n_k+\theta y_k^{n+1},\nonumber
	\end{align}
that is
\begin{equation}\label{eq:DO_ord1}
	\b y^{n+\theta}=(1-\theta)\b y^n+\theta \b y^{n+1}.
\end{equation}
Now since this is a convex combination of positive data, unconditional positivity can be seen immediately. 

With that, the one parameter family of second order MPRK schemes, MPRK22$(\alpha)$ with $\alpha\geq \frac12$ from \cite{KM18} can be equipped as follows.
\begin{align}
	y_k^{(1)} =&\, y_k^n,\nonumber\\ 
	y_k^{(2)} =&\, y_k^n + \alpha\Delta t\sum_{\nu=1}^N\left(p_{k\nu}(\b y^{(1)})\frac{y_\nu^{(2)}}{y_\nu^n}-d_{k\nu}(\b y^{(1)})\frac{y_k^{(2)}}{y_k^n}\right),\nonumber\\ 
	y_k^{n+1} =&\, y_k^n + \Delta t\sum_{\nu=1}^N\left( \Biggl(\biggl(1-\frac1{2\alpha}\biggr) p_{k\nu}(\b y^{(1)})+\frac1{2\alpha} p_{k\nu}(\b y^{(2)})\Biggr)\frac{y_\nu^{n+1}}{(y_\nu^{(2)})^{\frac{1}{\alpha}}(y_\nu^n)^{1-\frac{1}{\alpha}}}\right.\nonumber\\
	& \left. - \Biggl(\biggl(1-\frac1{2\alpha}\biggr) d_{k\nu}(\b y^{(1)})+ \frac1{2\alpha} d_{k\nu}(\b y^{(2)})\Biggr)\frac{y_k^{n+1}}{(y_k^{(2)})^{\frac{1}{\alpha}}(y_k^n)^{1-\frac{1}{\alpha}}}\right),\nonumber\\
		\b y^{n+\theta}&=(1-\theta)\b y^n+\theta \b y^{n+1}.\nonumber
\end{align}	
\subsection{Discussion of Higher Order Explicit Dense Output Formulae}
Explicit positivity preserving dense output formulae are also considered in \cite{KLJK17}, however, there strong-stability-preserving RK (SSPRK) schemes are equipped for which positivity is only guaranteed under some time step constraint. So, if we use the same formula for $p^*=2$, namely
\begin{equation}\label{eq:DO_bj_2nd}
	\bar b_1(\theta)=\theta -(1-b_1)\theta^2,\quad \bar b_j(\theta)=\theta^2b_j, \quad j=2,\dotsc,s,
\end{equation}
we may end up with a second order dense output formula (which is still to be proven), however, we cannot expect it to be unconditionally positive. Indeed, looking at a simple linear PDS
\begin{equation}\label{eq:test_IVP}
	\b y'(t)=\begin{pmatrix*}[r] -5 & 1\\5 & -1
	\end{pmatrix*} \b y(t),\quad \b y^0=(0.99,1)^T
\end{equation}
and approximate it with the third order MPRK43(1,0.5) scheme from \cite{KM18Order3}, we observe that the first component of the numerical solution is undershooting $0$, see Figure~\ref{fig:MPRK43IDO}.
\begin{figure}[htbp]
	\centering
	\includegraphics[width=.7\textwidth]{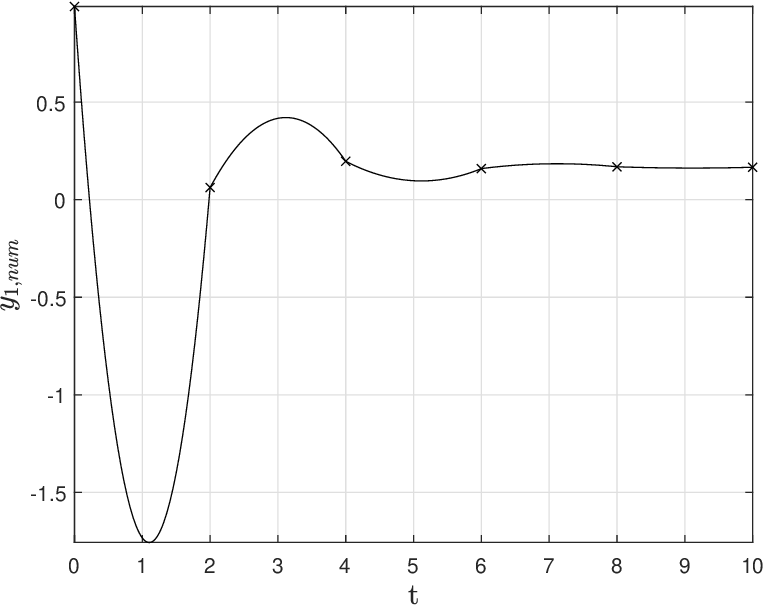}
	\caption{Numerical approximating $y_{1}^n$ of \eqref{eq:test_IVP} using MPRK43(1,0.5) from \cite{KM18Order3} with $\dt=2$, and the dense output using \eqref{eq:DO_ansatz1}, \eqref{eq:DO_bj_2nd}.}\label{fig:MPRK43IDO}
\end{figure}
We clearly see that even though the numerical approximation remains positive, the same is not true for its dense output. Also, since the formula \eqref{eq:DO_ansatz1} is conservative, \ie $y_1^{n+\theta} + y_2^{n+\theta}=y_1^0+y_2^0=1$, we deduce that $y_2^{n+\theta}$ is overshooting $1$. This problem also occurs when using higher order interpolating polynomials such as for the commonly used cubic Hermite interpolation \cite{HS2008}. Thus, one may try to fulfill sufficient conditions for the non-negativity of cubic Hermite polynomial. One such condition is derived in \cite{DEH1989}, which reads
\[ -3 \frac{y_j^n}{\dt}\leq f_j(\b y^n)\leq3 \frac{y_j^n}{\dt} \quad \forall j=1,\dotsc,N,\quad  n=0,1,\dotsc\]
for positive data.
Since a dense output formula should not change the computed numerical approximations, one now may be tempted to disturb the slopes $f_j(\b y^n)$ by some $\epsilon_{jn}=\O(\dt^p)$ using, for instance, a polynomial and additionally impose a conservativity constraint $\sum_{j=1}^N\epsilon_{jn}=0$ resulting in a linear optimization problem. However, the resulting problem may only possess a solution for $\dt$ small enough, but clearly none except $\O(\dt^p)=\epsilon_{jn}=-f_j(\b y^n)$  (\ie using slopes $0$) as $\dt\to\infty$ since $y_j^n$ is bounded due to the conservativity. However, $f_j(\b y^n)=\O(\dt^p)$ is not fulfilled in general. 

Finally, in view of \eqref{eq:DO_ord1}, the construction of dense output formulae with \eqref{eq:inner_consistency} as a convex combination of the stage vectors, the update $\b y^{n+1}$, and potentially the PWDs, is a valid candidate. However, this approach becomes more and more involved as the required order and number of stages increase, and thus, a search for such formulae has not been conducted.

Altogether, instead of using the explicit formula \eqref{eq:DO_ansatz1}, we rather propose an implicit formula which fits naturally into the MPRK framework. We update the formula \eqref{eq:DO_ansatz1} for $p^*\geq 2$ to
\begin{equation}\label{eq:DO_ansatz2}
	y^{n+\theta}_k = y^n_k + \dt \sum_{j=1}^s \bar b_j(\theta) \sum_{\nu=1}^N \left( p_{k\nu}(\byj)\frac{y^{n+\theta}_\nu}{\bar\sigma_\nu(\theta)} - d_{k\nu}(\byj) \frac{y^{n+\theta}_k}{\bar\sigma_k(\theta)} \right),
\end{equation}
where $\bar\bsigma(\theta)$ is yet to be determined. Now, this formula is linearly implicit and conservative, and as long as $\bar b_j(\theta)\geq 0$ and $\bar\sigma_\nu(\theta)>0$ for all $\nu=1,\dotsc,N$, the formula returns a positive output \cite{KM18}. Indeed, this is already the case in \eqref{eq:first_order} since $b_j\in[0,1]$ implies $\bar b_j(\theta)=\theta b_j\in [0,1]$.

If the positivity of $\bar b_j$ is not guaranteed for a given $\theta$, a positive and conservative approximation can be achieved by using an index function, see \cite{MPDeC} for the details. In particular, \eqref{eq:DO_ansatz2} becomes
\begin{equation}\label{eq:DO_ansatz2_index}
	y^{n+\theta}_k = y^n_k + \dt \sum_{j=1}^s \bar b_j(\theta) \sum_{\nu=1}^N \left( p_{k\nu}(\byj)\frac{y^{n+\theta}_{\delta(\nu,k,\bar b_j(\theta))}}{\bar\sigma_{\delta(\nu,k,\bar b_j(\theta))}(\theta)} - d_{k\nu}(\byj) \frac{y^{n+\theta}_{\delta(k,\nu,\bar b_j(\theta))}}{\bar\sigma_{\delta(k,\nu,\bar b_j(\theta))}(\theta)} \right),
\end{equation}
with the index function
\begin{equation}\label{eq:indexfun}
	\delta(\nu,k,x)=\begin{cases}
		\nu, &x\geq 0,\\
		k, & x< 0.
	\end{cases}
\end{equation}
Looking at \eqref{eq:DO_ansatz2}, a linear system has to be solved for any additional point in time for which an approximation is needed. However, this differs from applying the method with smaller $\dt$ since the formula \eqref{eq:DO_ansatz2} uses the same stage vectors, which only need to be calculated once.

Similarly to the proof of Lemma~\ref{lem:sigma}, we see that $\bar\bsigma(\theta)$ needs to be a $(p-1)$-th order approximation to $\b y(t^n+\theta \dt)$ by using the condition $\sum_{j=1}^s\bar b_j(\theta)=\theta$ for a first order RK dense output formula. This is particularly the reason why we introduced its dependency on $\theta$ in the first place.

The key idea to achieve a $(p-1)$-th order approximation to $\b y(t^n+\theta \dt)$ is to use a lower order dense output formula, which opens up the door for our boot-strapping process. Also note that the order conditions \eqref{eq:u=theta:gamma} remain the same and only \eqref{eq:bjtheta_exp} is replaced by
\begin{equation}\label{eq:bjtheta_imp}
	b_i^{[\nu]}(\b y^n,\dt,\theta)=\bar b_j(\theta)\frac{y^{n+\theta}_\nu}{\bar\sigma_\nu(\theta)}.
\end{equation}
We want to note at this point that we assume $	b_i^{[\nu]}(\b y^n,\dt,0)=0$ as well as $	b_i^{[\nu]}(\b y^n,\dt,1)=	b_i^{[\nu]}(\b y^n,\dt)$
 for the inner consistency.
\section{Boot-Strapping Process for Higher Order Dense Output}
For the following analysis we assume that $\bar\bsigma(\theta)>\b 0$ is a continuous function of $\b y^n$ and the stages. Then \cite[Lemma 4.6]{IzginThesis}  implies $\b y^{n+\theta}=\O(1)$ as $\dt \to 0$ and $\frac{y^{n+\theta}_\mu}{\bar\sigma_\mu(\theta)}=\O(1)$ as $\dt \to 0$. Furthermore, \cite[Lemma 4.8]{IzginThesis} justifies the implication
\[\bar\bsigma(\theta)=\b y^{n+\theta} +\O(\dt^k) \Longrightarrow \frac{y^{n+\theta}_\mu}{\bar\sigma_\mu(\theta)}=1+\O(\dt^k),\quad \mu=1,\dotsc,N.\] We will use these results without further notice. 
Also, a significant result simplifying the construction of a dense output formula is the following variant of \cite[Corollary 4.3]{IzginThesis}, which is not restricted to MPRK schemes let alone a specific form of $b_i^{[\nu]}(\b y^n,\dt,\theta)$.
\begin{lemma}\label{lem:suff_cond_dense} Let $\b A,\b b,\b c$ define an RK method of order $\hat p\geq 1$ and let \eqref{eq:RKdense} be a dense output formula of order $\hat p^*=\max\{\hat p-1,1\}$.
	If  \begin{equation}\label{eq:lem_suff}
		a_{ij}^{[\nu]}(\b y^n,\dt)=a_{ij}+\O(\dt^{\hat p^*-1}) \qta b_j^{[\nu]}(\b y^n,\dt,\theta)=\bar b_j(\theta)+\O(\dt^{\hat p^*}),
	\end{equation} for $i,j=1,\dotsc,s$ and $\nu=1,\dotsc,N$, then the method 
	\begin{equation}\label{eq:nsrk}
		\begin{aligned}
			\byi & = \b y^n + \dt \sum_{j=1}^s  \sum_{\substack{\nu=1}}^N a^{[\nu]}_{ij}(\b y^n,\dt)  \fnu(\byj), \quad i=1,\dotsc,s,\\
			\b y^{n+\theta} & = \b y^n + \dt \sum_{j=1}^s \sum_{\substack{\nu=1}}^N b^{[\nu]}_j(\b y^n,\dt,\theta) \fnu(\byj).
		\end{aligned}
	\end{equation} satisfies $\b y^{n+\theta}=\NB^\theta_{\hat p^*}(\frac1\gamma,\b y^n)+\O(\dt^{\hat p^*+1})$.
\end{lemma}
\begin{proof}
	Since we assumed $b^{[\nu]}_j(\b y^n,\dt,1)=b^{[\nu]}_j(\b y^n,\dt)$ for inner consistency, we see from \cite[Corollary 4.3]{IzginThesis} that $\b y^{n+1}=\NB_{\hat p^*}(\frac1\gamma,\b y^n)+\O(\dt^{\hat p^*+1})$. Indeed, along the same lines it can be seen that $\b y^{n+\theta}=\b y^n+\sum_{\tau\in NT_{\hat p^*}}\frac{\dt^{\lvert \tau\rvert}}{\sigma(\tau)}u(\tau,\theta) \dF(\tau)(\b y^n)+\O(\dt^{\hat p^*+1})$, where $u$ is obtained from \eqref{eq:pertcond} by replacing $a_{ij}^{[\nu]}(\b y^n,\dt)$ by $a_{ij}$ and $b_j^{[\nu]}(\b y^n,\dt,\theta)$ by $\bar b_j(\theta)$. Now, since the dense output formula of the RK scheme is assumed to be of order $\hat p^*$, we see $\b y^{n+\theta}=\NB^\theta_{\hat p^*}(\frac1\gamma,\b y^n)+\O(\dt^{\hat p^*+1})$.
\end{proof}
We used $\hat p$ and $\hat p^*$ in this lemma because we will not use it with $\hat p=p$ and $\hat p^*=p^*$ as will be seen in the proof of Theorem~\ref{thm:sigma}.
This lemma needs three ingredients to return the requested dense output formula for the MPRK scheme. First, a dense output formula for the underlying RK scheme returning $\bar b_j(\theta)$, which may be done by a boot-strapping process. Second, an MPRK method of a certain order. Third, the condition \eqref{eq:lem_suff} needs to be fulfilled. Luckily, all MPRK methods of order $p\in \{2,3,4\}$ with an underlying RK scheme with $s=p$ stages satisfy $a_{ij}^{[\nu]}(\b y^n,\dt)=a_{ij}+\O(\dt^{\max\{p-2,1\}})$, \ie $\frac{y_k^{(i)}}{\pi_k^{(i)}}=1+\O(\dt^{\max\{p-2,1\}})$, see \cite[Theorems~4.12,~4.13,~4.15]{IzginThesis}. Altogether, this motivates us the assumption of the following result.
\begin{theorem}\label{thm:sigma}
Let the RK scheme \eqref{eq:RK} be of order $p\geq 2$, equipped with a dense output formula \eqref{eq:RKdense} of order $p^*=p-1$. Furthermore, let the corresponding MPRK scheme \eqref{eq:MPRK} be of order $p$ and assume $\frac{y_k^{(i)}}{\pi_k^{(i)}}=1+\O(\dt^{\max\{p-2,1\}})$. 
If $\bar\bsigma(\theta)=\NB^\theta_{p^*-1}(\frac{1}{\gamma},\b y^n)+~\O(\dt^{p^*})$, then the dense output formula \eqref{eq:bjtheta_imp} has a convergence rate of order $p$.
\end{theorem}
\begin{proof}
Every assumption of Lemma~\ref{lem:suff_cond_dense} but \[ b_j^{[\nu]}(\b y^n,\dt,\theta)=\bar b_j(\theta)+\O(\dt^{\hat p})\] is satisfied for any $1\leq\hat p\leq p-1=p^*$. Thus, if we even proved \[\frac{y^{n+\theta}_\mu}{\bar\sigma_\mu(\theta)}=1+\O(\dt^{p^*}),\] Lemma~\ref{lem:suff_cond_dense} would imply that $\b y^{n+\theta}=\NB^\theta_{p^*}(\frac1\gamma,\b y^n)+\O(\dt^{p^*+1})$. Since $p^*+1= p$ we could then deduce from the order $p$ of the MPRK method and \cite[Section II.6]{HNW1993} that the dense output formula is convergent of order $p$.

Now, we have already discussed that $\frac{y^{n+\theta}_\mu}{\bar\sigma_\mu(\theta)}=\O(1)$. Introducing this into \eqref{eq:DO_ansatz2}, we see $\b y^{n+\theta} = \b y^n+\O(\dt)=\NB_0^\theta(\frac1\gamma,\b y^n) +\O(\dt)$. 
From this we conclude $\b y^{n+\theta}=\bar\bsigma(\theta) +\O(\dt)$, which implies $\frac{y^{n+\theta}_\mu}{\bar\sigma_\mu(\theta)}=1+\O(\dt)$. If $p^*>1$, we use Lemma~\ref{lem:suff_cond_dense} with $\hat p=1$ to receive $\b y^{n+\theta} =\NB_1^\theta(\frac1\gamma,\b y^n) +\O(\dt^2)$ and thus $\frac{y^{n+\theta}_\mu}{\bar\sigma_\mu(\theta)}=1+\O(\dt^2)$ by the same reasoning as above. By induction we deduce $\frac{y^{n+\theta}_\mu}{\bar\sigma_\mu(\theta)}=1+\O(\dt^{p^*})$. 

\end{proof}
This theorem now puts us in the position to construct a dense output formula for MPRK schemes up to order four and to describe our boot-strapping technique. Interpreting MPDeC methods from \cite{MPDeC} as MPRK schemes and to develop dense output formulae for arbitrary high-order MPDeC schemes is outside the scope of this work and left for future work.

We already have a first order dense output formula at our disposal. Hence, we continue constructing a second order formula. 
\subsection{Second Order Dense output}
We apply Theorem~\ref{thm:sigma} using $p=3$  looking at a third order MPRK scheme, so that with $p^*=2$ we seek $\bar\bsigma(\theta)$ to be merely a first order dense output formula. To that end, we simply use our first order dense output formula setting 
\begin{equation}\label{eq:sigma(theta)_1}
\bar\bsigma(\theta)=(1-\theta)\b y^n+\theta \bsigma.
\end{equation}
Since we also have a second order dense output formula for the underlying RK scheme, \eg using \eqref{eq:DO_bj_2nd}, we end up with the following result.
\begin{theorem} Consider an MPRK scheme \eqref{eq:MPRK} of order $p=3$ based on an $3$-stage RK scheme of order $3$. Then, using $\bar\bsigma(\theta)$ from \eqref{eq:sigma(theta)_1} and $\bar b_j(\theta)$ from \eqref{eq:DO_bj_2nd}, the formula \eqref{eq:bjtheta_imp} is convergent of order three.
\end{theorem}
There are two families of third order MPRK schemes mentioned in \cite{KM18Order3} which can be equipped with this formula. Moreover, even the third order MPDeC method satisfies the assumption $\frac{y_k^{(i)}}{\pi_k^{(i)}}=1+\O(\dt)$ (after adapting the notation)  due to \cite[Lemmas~4.9,~4.10]{MPDeC}, and thus can be equipped with the same formula. 
The corresponding dense output formula for the schemes in \cite{KM18Order3} can be written as 
\begin{align}
	y^{(1)}_k &= y^n_k,\nonumber\\
	y^{(2)}_k &= y^n_k
	+  a_{21} \dt \sum_{\nu=1}^N\left(
	p_{k\nu}\bigl( \b y^n \bigr) \frac{y^{(2)}_\nu}{y^n_\nu}
	- d_{k\nu}\bigl( \b y^n \bigr) \frac{y^{(2)}_k}{y^n_k}
	\right),
	\nonumber\\
	y^{(3)}_k &= y^n_k
	+\Delta t \sum_{\nu=1}^N
	\Biggl(\left(a_{31} p_{k\nu}\bigl(\b y^n\bigr)+ a_{32} p_{k\nu} \bigl(\b y^{(2)}\bigr) \right)  \frac{  y_\nu^{(3)}
	}{\bigl(y_\nu^{(2)}\bigr)^{\frac1p } \bigl(y_\nu^n\bigr)^{1-\frac1p} }
	\nonumber\\
	& \qquad\qquad\qquad
	-\left(a_{31} d_{k\nu}\bigl(\b y^n\bigr)+ a_{32} d_{k\nu} \bigl(\b y^{(2)}\bigr) \right)  \frac{  y_k^{(3)}
	}{\bigl(y_k^{(2)}\bigr)^{\frac1p } \bigl(y_k^n\bigr)^{1-\frac1p} }
	\Biggr),
	\nonumber\\
	\hphantom{baselineskip}&\nonumber\\
	\sigma_k &= y_k^n + \Delta t \sum_{\nu=1}^N
	\Biggl(\left( \beta_1 p_{k\nu} \bigl( \b y^n\bigr) +\beta_2 p_{k\nu} \bigl(\b y^{(2)}\bigr)  \right) \frac{\sigma_\nu}{\bigl(y_\nu^{(2)} \bigr)^{\frac1q}
		\bigl(y_\nu^n\bigr)^{1-\frac1q}}
	\nonumber\\ & \qquad-
	\left( \beta_1 d_{k\nu} \bigl( \b y^n\bigr) +\beta_2 d_{k\nu} \bigl(\b y^{(2)}\bigr)  \right) \frac{\sigma_k}{\bigl(y_k^{(2)} \bigr)^{\frac1q}
		\bigl(y_k^n\bigr)^{1-\frac1q}}\Biggr), \label{eq:sigmaMPRK43}\\
	\bar\bsigma(\theta)&=(1-\theta)\b y^n+\theta \bsigma,\nonumber \\
	y^{n+\theta}_k &= y^n_k
	+ \dt \sum_{\nu=1}^N \Biggl(
	\left( (\theta -(1-b_1)\theta^2)p_{k\nu}\bigl( \b y^n \bigr) +\theta^2b_2p_{k\nu}\bigl( \b y^{(2)} \bigr)
	+\theta^2 b_3 p_{k\nu}\bigl( \b y^{(3)} \bigr)
	\right) \frac{y^{n+\theta}_\nu}{\bar \sigma_\nu(\theta)}
	\nonumber\\&\qquad\qquad\qquad
	- \left( (\theta -(1-b_1)\theta^2) d_{k\nu}\bigl( \b y^n \bigr) +\theta^2b_2d_{k\nu}\bigl( \b y^{(2)} \bigr)
	+ \theta^2b_3 d_{k\nu}\bigl( \b y^{(3)} \bigr)
	\right) \frac{y^{n+\theta}_k}{\bar\sigma_k(\theta)}
	\Biggr),	\label{eq:MPRK43-family}
\end{align}
where $p=3 a_{21}\left(a_{31}+a_{32} \right)b_3,\; q=a_{21},\;\beta_2=\frac{1}{2a_{21}}$ and $\beta_1= 1-\beta_2$.
\subsection{Third Order Dense Output}
Since the condition $\tfrac{y_k^{(i)}}{\pi_k^{(i)}}=1+\O(\dt^{p-2})$ for the fourth order MPDeC method from \cite{MPDeC} is not yet investigated, we focus from now on MPRK schemes, for which this condition is proven, see \cite[Theorem 4.15]{IzginThesis}. First, we recall corresponding the dense output formula derived in \cite{Zennaro1986}, \ie
\[\bar b_1(\theta)=2(1-4b_1)\theta^3+3(3b_1-1)\theta^2+\theta,\quad \bar b_i(\theta)=4(3c_i-2)b_i\theta^3+3(3-4c_i)b_i\theta^2, i=2,3,4.\]
The fourth order MPRK scheme from \cite{NSARK,IzginThesis} is based on the classical RK scheme of order $4$ described by
	\begin{equation*}
	\begin{aligned}
		\def\arraystretch{1.2}
		\begin{array}{c|cccc}
			0 &  & & & \\
			\frac12 & \frac12 & & &\\
			\frac12 &0 &\frac12 & &\\
			1 &0 & 0 &1 &\\
			\hline
			& \frac16 &\frac13&\frac13 & \frac16
		\end{array}
	\end{aligned}
\end{equation*}
Thus, we have
\begin{equation}
	\begin{aligned}
		\bar b_1(\theta)=\tfrac23\theta^3-\tfrac32\theta^2+\theta,\quad
		\bar b_2(\theta) = 	\bar b_3(\theta) =-\tfrac23\theta^3+\theta^2,	\quad	 \bar b_4(\theta)=\tfrac23\theta^3-\tfrac12\theta^2,
	\end{aligned}
\end{equation}
for which $\bar b_i(\theta)\geq 0$ for $i=1,2,3$, however, for instance $\bar b_4(\tfrac12)<0$. As a result of this, we need to introduce the index function into the dense output formula \eqref{eq:DO_ansatz2} which results in \eqref{eq:DO_ansatz2_index}. Now this shows that even though the above Butcher tableau is non-negative, the corresponding dense output tableau is not. In recent works such as \cite{izgin2023nonglobal, IssuesMPRK} inferior stability properties for such schemes are discovered rising the question of whether this formula yields another such example. However, the investigation of this question is outside the scope of this work.

Now, the fourth order MPRK scheme is constructed as follows. Due to \eqref{eq:DO_ansatz2_index}, we only need to specify $\bm \pi^{(i)}$ for $i=2,3,4$  and $\bssigma$. First, the third order MPRK scheme (\ie \eqref{eq:MPRK43-family} with $\theta =1$) is used to compute $\bsigma$. Secondly, the embedded second order method \eqref{eq:sigmaMPRK43} is used with time steps $c_i\dt$ is used to compute $\bm \pi^{(i)}$ for $i=2,3,4$ resulting in $\frac{\yi_\nu}{\pi_\nu^{(i)}}=1+\O(\dt^3)$, see \cite{IzginThesis} for more details.

Now, since the third order MPRK scheme is used to compute $\bsigma$ of the fourth order MPRK scheme, we can simply use the corresponding second order dense output formula to define $\bar \bsigma(\theta)$ of this fourth order method. Furthermore, the fourth order MPRK scheme is constructed satisfying the sufficient conditions of \cite[Corollary 4.3]{IzginThesis}, and hence, the assumption $\tfrac{y_k^{(i)}}{\pi_k^{(i)}}=1+\O(\dt^{p-2})$ of Theorem~\ref{thm:sigma} is naturally satisfied.  

Since the overall MPRK method has itself 10 stages, we will not write out the dense output formula out.

\subsection{Boot-Strapping Process}
The key observation here is to use the lower order dense output formula to define $\bar\bsigma(\theta)$ of the new, higher order method. With this, the boot-strapping process reduces to use known dense output formulae of an RK scheme and to check the condition $\frac{y_k^{(i)}}{\pi_k^{(i)}}=1+\O(\dt^{p-2})$ (looking at $p\geq 3$). We note that the latter is always fulfilled, if the MPRK scheme is constructed using the sufficient condition stated in \cite[Corollary~4.3]{IzginThesis}, first introduced in \cite{NSARK}. We also observed that this condition is even necessary for $p\in \{2,3,4\}$. However, a discussion of whether this condition is necessary for every $p\geq 2$ is still an open research topic.

\section{Summary and Outlook}
In this work we have developed a boot-strapping technique to equip modified Patankar--Runge--Kutta (MPRK) methods with a dense output formula of appropriate accuracy. We have stated the corresponding order conditions for the formula and successively constructed formulae for MPRK schemes up to order four. There, the first order dense output formula is explicit while the remaining ones are linearly implicit. Still, these formulae are unconditional positive and conservative, which is the natural requirement we imposed since the MPRK schemes have this property. In addition, the additional computational effort is still less than using the method with a smaller step size since the stage vectors only need to be computed once. We have also discussed the possibility and issues of different approaches for designing a dense output formula. However, the presented approach involving linearly implicit formulae has the advantage of being generalized easily also for different Patankar-type schemes such as modified Patankar Deferred Correction (MPDeC) methods. Indeed, the we found that the first and second order dense output formula can be used to equip second and third order MPDeC schemes, respectively. The discussion of higher order MPDeC schemes is left for future works. To that end, the investigation of the property $\frac{y_k^{(i)}}{\pi_k^{(i)}}=1+\O(\dt^{p-2})$ for $p\geq 3$ is of interest.

Furthermore, we demonstrated that even though the MPRK scheme may be based on a non-negative Butcher tableau, the corresponding dense output formulae may result in negative values $\bar b_j(\theta)$ for some $\theta\in[0,1]$ necessitating the use of the index function \eqref{eq:indexfun}. Now, since schemes based on a partially negative Butcher tableau showed inferior stability properties \cite{IssuesMPRK,izgin2023nonglobal}, further investigations of such formulae is needed and left as a future research topic.
\section*{Acknowledgments}
The author T.\ Izgin gratefully acknowledges the financial support by the Deutsche Forschungsgemeinschaft (DFG) through the grant ME 1889/10-1 (DFG project number 466355003).

\bibliography{cas-refs}

\end{document}